\documentclass[oneside,12pts]{article}
 \usepackage{etex}

\usepackage{microtype}
\usepackage[a4paper]{geometry}
\usepackage[dvips]{epsfig}
\usepackage[x11names]{xcolor} 
\usepackage{tikz}
\usetikzlibrary{calc}
\usetikzlibrary{matrix}
\usetikzlibrary{arrows}
\usepackage{tikz-qtree}
\usepackage{pgfplots}
\usepackage{pifont}
\usepackage{dsfont}
\usepackage{graphicx}
\usepackage{frcursive}
\usepackage[absolute]{textpos}
\usepackage{moreverb} 
\usepackage{color}
\usepackage{amssymb}
\usepackage[colorlinks=true]{hyperref}
\usepackage{float}
\usepackage{ae}
\usepackage{placeins}
\usepackage{amsfonts, amsmath} 
\usepackage[title,titletoc]{appendix}
\usepackage{comment}
\usepackage[utf8]{inputenc}
\usepackage[english]{babel}
\usepackage{amsthm}
\usepackage{comment}

\begin{document}
\setcounter{page}{1}
\title{
 On local quasi efficient solutions for nonsmooth vector optimization
}

\author{Mohsine Jennane$^1$ \and  Lhoussain El Fadil$^1$ \and El Mostafa Kalmoun$^2$}

\date{%
$^1$ Department of Mathematics, Faculty of Sciences Dhar-mahraz\\ Sidi Mohammed Ben Abdellah University, Fez, Morocco\\
E-mail: mohsine.jennane@usmba.ac.ma; lhouelfadil2@gmail.com\\[1ex]
$^2$ Department of Mathematics, Statistics and Physics\\
College of Arts and Sciences, Qatar University, Doha, Qatar\\
E-mail:ekalmoun@qu.edu.qa
}

\maketitle

\begin{abstract}
We are interested in local quasi efficient solutions for nonsmooth vector optimization problems under new generalized approximate invexity assumptions. We formulate necessary and sufficient optimality conditions based on Stampacchia and Minty types of vector variational inequalities involving
	Clarke's generalized Jacobians. We also establish the relationship between local quasi weak efficient solutions and vector critical points.

\bigskip

\end{abstract}

\newtheorem{theorem}{Theorem}[section]
\newtheorem{lemma}[theorem]{Lemma}
\newtheorem{Definition}[theorem]{Definition}
\newtheorem{Proposition}[theorem]{Proposition}
\newtheorem{Remark}[theorem]{Remark}
\newtheorem{Remarks}[theorem]{Remarks}
\newtheorem{corollary}[theorem]{Corollary}
\newtheorem{Example}[theorem]{Example}

\numberwithin{equation}{section}

\section{Introduction}

Throughout the paper, $\mathbb{R}^n$ and $\mathbb{R}^m$ are the n-dimensional and the m-dimensional Euclidean spaces respectively,
$X\subseteq \mathbb{R}^n$ is a nonempty open set and $f := (f_1,..., f_m):X \rightarrow \mathbb{R}^m$ is a vector-valued function. Assume $C\subseteq \mathbb{R}^m$ is a closed pointed convex cone with a nonempty interior. Partial ordering is defined on $X$ using $C$ as follows: $x \geq_C y$ (resp. $x >_C y$) if $x - y \in C$  (resp. $x - y \in int\, C$).

Let us consider the following vector optimization problem (VOP):
$$ \min_{x \in X} f (x).$$
Recall that $\xi\in X$ is an efficient solution of (VOP), if no other feasible vector
$x \in X$ satisfies $f (x) \leq_C f (\xi)$.  Suppose we are given a mapping $\eta :X \times X \rightarrow \mathbb{R}^n$ and a vector $e>_C0$. In this paper, we are interested in optimality conditions for a weaker type of solutions to (VOP); namely, local $(\eta,e)$-quasi efficient solutions.
\begin{Definition}
	We say that a feasible point $\xi\in X$ is
	a local $(\eta,e)$-quasi (weak) efficient solution of (VOP) if there is $r > 0$ such that there is no $x \in B(\xi; r)$ satisfying
	$f(x)+e \|\eta(x,\xi)\|  \: \leq_C \ (<_C) \; f (\xi).$
\end{Definition}

	Vector optimization problems have a number of important applications in applied science, engineering and economics; see for example~\cite{eichfelder2012vector} and references therein. A powerful tool to study their optimality conditions is through vector variational inequalities~\cite{ansari2018vector}, initiated by Giannessi \cite{Giannessi1980TheoryandApplications} to be vector extensions of Stampacchia variational inequalities~\cite{Stampacchia1960AS}. For differentiable and convex  multiobjective functions, Giannessi~\cite{Giannessi1998NTMP} used vector variational inequalities of Minty type~\cite{Minty1967BAM} to derive necessary and sufficient conditions for efficient solutions. Generalizations of these optimality conditions were given for different types of generalized convexity
	\cite{Yang2004JOTA,Gang2008CMA,Fang2009Opt,Homidan2010JOTA,Zafarani2012JGO} and generalized invexity \cite{Long2012Optim,Mishra2006NA,Yang2006Opt,Ansari2013Opt}. In~\cite{MishraUpadhyay2013Positivity}, relationships between quasi efficient points, solutions to Stampacchia vector variational inequalities and vector critical points were identified under approximate convexity assumptions. On the other hand, a new concept of approximate invexity was defined in \cite{NoorMishraMomani2005NAFJ} as an extension of approximate convexity \cite{NgaiLucThera2000JNCA}. Furthermore, four new classes of generalized convexity were introduced in \cite{BhatiaGupta2013OptimLett} as a generalization of the classical notions of pseudoconvexity and quasiconvexity.

	It is worth mentioning that in case of nonsmooth vector optimization, the appropriate tool to study optimality conditions is Clarke's generalized Jacobian~\cite{clarke1983optimization} when the multiobjective function is supposed to be locally Lipschitz. Recall that $f$ is locally Lipschitz if for any $x_0 \in X$ there are two positive reals $k$ and $r > 0$ with
	$$\| f (x) - f (y)\| \leq k \|x-y\|, \quad \forall x,y \in B(x_0, r).$$
	In this case, Clarke's generalized
	Jacobian~\cite{clarke1983optimization} of $f$ at $x \in X$ is the set of $m\times n$ matrices defined by
	\begin{equation}\label{eq:clarke}
	\partial f (x) = co\{\underset{i \rightarrow +\infty}{lim} Jf(x^{(i)}): x^{(i)}\rightarrow x, x^{(i)} \in S\},
	\end{equation}
	where $co$ indicates the convex hull, $Jf(x^{(i)})$ is the Jacobian of $f$ at $x^{(i)}$, and $S$
	is the differentiability set of $f$.
	
	In many previous works~\cite{Mishra2006NA,GuptaMishra2018Optimization,MishraUpadhyay2013Positivity,Gut2016OL},  Clarke's generalized Jacobian of $f$ at $x$ was defined to be the Cartesian product of its real-valued components Clarke's subdifferentials $\partial f_1 (x)\times ... \times \partial f_m (x)$.
	Using \eqref{eq:clarke} in vector optimizations problems seems to be a more natural extension of the real-valued case since $\partial f$ as defined above is not equal to this Cartesian product. Nevertheless, the inclusion
	$$\partial f (x)\subseteq \partial f_1 (x)\times ... \times \partial f_m (x)$$
	reveals that using $\partial f$ in the generalized convexity/invexity definitions and in the vector variational inequalities appear to be less restrictive than the Cartesian product.

	Our aim in this paper is to introduce new types of generalized
	approximate invexity and investigate their use in deriving optimality conditions for local $(\eta,e)$-quasi (weak) efficient solutions of (VOP).  In particular, we use both strong and weak forms of Stampacchia and Minty vector variational inequalities given in terms of Clarke's generalized Jacobian \eqref{eq:clarke}. Finally, we show also the relationship between local $(\eta,e)$-quasi weak efficient solutions and vector critical points.

\section{Generalized approximate invexity}
		From now onward, we suppose that 
	$f$ is locally Lipschitz. Let us present our extensions of the generalized approximate convexity concepts provided in~\cite{NgaiLucThera2000JNCA,BhatiaGupta2013OptimLett,noor2006JMAA}.
	\begin{Definition}  $f$ is said to be approximate $(\eta,e)$-invex at $x_0 \in X$ if there is $r > 0$
		such that for any  $x, y \in B(x_0,r)$,
		$$f(x)-f(y) \geq_C A_y \eta(x,y) - e \| \eta(x,y)\|, \; \forall A_y \in \partial f(y).$$
		$f$ is said to be approximate $(\eta,e)$-invex on $X$, if $f$ is approximate $(\eta,e)$-invex at each $x_0\in X$.
	\end{Definition}
	By taking $\eta(x,y)=x-y$, we deduce that approximate convexity \cite{NgaiLucThera2000JNCA} is a special case of approximate invexity. However,  the following counter-example shows the converse is generally not true.
	\begin{Example}
		Let $X=\mathbb{R}$, $C=\mathbb{R}^2_+$ and for $x,y\in \mathbb{R}$
		$$
		f(x)=(x,\varphi(x))^T \quad \text{ where }\quad
		\varphi(x)=\begin{cases}  4x-x^2, & x\geq 0 \\
		2x, & x<0;
		\end{cases}$$
		and
		$$ \eta(x,y)= -|x-y|.$$
		Clarke's generalized Jacobian of $f$ at $x$ is given by
		$$\partial f(x)= \begin{cases}
		\{(1,4-2x)^T\}, & x>0; \\
		\{(1,k)^T:k\in[2,4]\}, & x=0; \\
		\{(1,2)^T\}, & x<0.
		\end{cases}$$
		Let $x_0=0$, $e=(\varepsilon,\varepsilon)$ for an arbitrary real $\varepsilon >0$, and take $r=min(1,\frac{\varepsilon}{2}) >0$.
		For all $x,y \in B(x_0, r)$ and all $A_y \in \partial f(y)$, we have
		$$f(x)-f(y)=(x-y,\varphi(x)-\varphi(y))^T,$$
		where
		$$\varphi(x)-\varphi(y)=\begin{cases}
		(x-y)(4-x-y), & \mbox{ if } y>0,x>0;\\
		2x-4y+y^2, & \mbox{ if } y>0,x\leq0;\\
		4x-x^2-2y, & \mbox{ if } y<0,x\geq0;\\
		2(x-y), & \mbox{ if } y<0,x<0;\\
		4x-x^2, & \mbox{ if } y=0,x>0;\\
		2x, & \mbox{ if } y=0,x<0;
		\end{cases}$$
		and
		$$ A_y\eta(x,y) - e\|\eta(x,y)\|= ((-1-\varepsilon)|x-y|,\alpha(x,y))^T,$$
		where
		$$\alpha(x,y)=\begin{cases}
		|x-y|(2y-4-\varepsilon), & \mbox{ if } y>0,x>0; \\
		2x-4y+y^2+x(2-y)+(x-y)(\varepsilon-y), & \mbox{ if } y>0,x\leq0;\\
		(x-y)(2y-4-\varepsilon), & \mbox{ if } y<0,x\geq0;\\
		|x-y|(-2-\varepsilon), & \mbox{ if } y<0,x<0; \\
		|x|(-k-\varepsilon), & \mbox{ if } y=0.\\
		\end{cases}$$
		We can easily verify that $f(x)-f(y)\geq_C A_y\eta(x,y) - e\|\eta(x,y)\|.$
		Hence $f$ is approximate $(\eta,e)$-invex at $x_0=0$.\\
		However, $f$ is not approximate convex. Indeed, if we take $y < 0$ and
		$x = 0$, then for $0 < \varepsilon < 1$ the inequality of approximate convexity is not satisfied.
	\end{Example}
	\begin{Definition} Let $x_0 \in X$. The function $f$ is said to be
		\begin{itemize}
			\item
			 approximate pseudo $(\eta,e)$-invex of type I at $x_0$
			if there exists $r >0$ so that for any $x, y \in B(x_0, r)$,
			$$f(x)-f(y) <_C - e\|\eta(x,y)\| \quad \Rightarrow \quad A_y\eta(x,y) <_C 0, \;\forall A_y \in \partial f(y);$$
	\item
		approximate pseudo $(\eta,e)$-invex of type II at $x_0$
		if there exists $r >0$ so that for any $x, y \in B(x_0, r)$,
		$$f(x) - f(y) <_C 0  \quad \Rightarrow \quad A_y\eta(x,y) + e\|\eta(x,y)\| <_C 0,
		\; \forall A_y \in \partial f(y).$$
	\item
	    approximate quasi $(\eta,e)$-invex of type I at $x_0$
		if there exists $r >0$ so that for any $x, y \in B(x_0, r)$,
		$$\exists  A_y \in \partial f(y):\; A_y\eta(x,y) -e\|\eta(x,y)\| >_C 0   \quad \Rightarrow \quad  f(x)-f(y)>_C 0.$$
	\item
		approximate quasi $(\eta,e)$-invex of type II at $x_0$
		if there exists $r >0$ so that for any $x, y \in B(x_0, r)$,
		$$\exists  A_y \in \partial f(y):\; A_y \eta(x,y)  >_C 0  \quad \Rightarrow \quad
		f(x) >_C f(y)+e\|\eta(x,y)\|.$$
		\end{itemize}	
	\end{Definition}
	\begin{Remark}
		\begin{itemize}
			\item If $f$ is approximate pseudo (resp. quasi) $(\eta,e)$-invex of type II at $x_0\in X$, then $f$ is approximate pseudo (resp. quasi) $(\eta,e)$-invex of type I at $x_0$.
			\item It is easy to see that any approximate $(\eta,e)$-invex function at $x_0$ is approximate
			pseudo $(\eta,e)$-invex function of type I and approximate quasi $(\eta,e)$-invex function of type I at $x_0$.
			\item There is no relation between approximate
			pseudo $(\eta,e)$-invex functions of type II and approximate quasi $(\eta,e)$-invex functions of type II and approximate invex functions (see \cite{GuptaMishra2018Optimization}).
		\end{itemize}
	\end{Remark}

	\section{Sufficient conditions for local quasi efficient solutions}
	
	We first consider Stampacchia and Minty types of vector variational inequalities involving Clarke’s generalized Jacobians:\\
	(SVVI) Find  $\xi\in X$ for which there exists no $x \in X$ satisfying
	$$A_{\xi} \eta(x,\xi)  \leq_C 0,\quad \forall A_{\xi} \in \partial f(\xi).$$
	(MVVI) Find $\xi\in X$ for which there exists no $x \in X$ satisfying
	$$A_{x} \eta(x,\xi)  \leq_C 0,\quad \forall A_{x} \in \partial f(x).$$
	
	We present sufficient conditions for local $(\eta,e)$-quasi efficient solutions of (VOP) under approximate invexity assumptions.
	\begin{theorem}\label{thm1} Suppose $f$ is approximate $(\eta,e)$-invex at $\xi\in X$. If $\xi$ is a solution of (SVVI), then $\xi$ is also a local $(\eta,e)$-quasi efficient solution of (VOP).
	\end{theorem}
	\begin{proof}
		Assume $\xi$ fails to be a local $(\eta,e)$-quasi efficient solution of (VOP). Hence for each
		$r> 0$ there is $x_0 \in B(\xi,r)$ satisfying
		\begin{equation}\label{eq:thm1}
		f(x_0)-f (\xi) \leq_C -e\|\eta (x_0,\xi)\|.
		\end{equation}
		Since $f$ is approximate $(\eta,e)$-invex at $\xi$, it follows that
		$$f(x_0)-f (\xi)\geq_C A_{\xi}\eta(x_0,\xi)  - e \|\eta (x_0,\xi)\|, \quad \forall   A_{\xi} \in \partial f(\xi).$$
		Using \eqref{eq:thm1}, we get
		$$ A_{\xi}\eta(x_0,\xi) \leq_C  0,\quad \forall A_{\xi} \in \partial f(\xi).$$
		This means $\xi$ does not solve (SVVI).
	\end{proof}

\begin{Remark}
	As the approximate invexity assumption is more general than approximate convexity, Theorem \ref{thm1} extends Theorem 3.1 in~\cite{MishraUpadhyay2013Positivity}.
	\end{Remark}
	
	The following theorem gives the conditions for a point to be a local $(\eta,e)$-quasi efficient solution of (VOP) in terms of (MVVI).
	
	\begin{theorem}\label{thm2} Suppose $-f$ is approximate $(\eta,e)$-invex  at $\xi\in X$ such that $\eta(x, \xi)+\eta(\xi, x)=0$ for all $x\in X$. If $\xi$ is a solution of (MVVI), then $\xi$ is also a local $(\eta,e)$-quasi efficient solution of (VOP).
	\end{theorem}
	
	\begin{proof}
		Assume the vector $\xi$ fails to be a local $(\eta,e)$-quasi efficient solution of (VOP). Thus for each
		$r> 0$ there is $x_0 \in B(\xi,r)$ satisfying \eqref{eq:thm1}. Again the  approximate $(\eta,e)$-invexity of $-f$ at $\xi$ yields
		$$(-f)(\xi)- (-f)(x_0)\geq_C A_{x_0} \eta (\xi,x_0)  - e \|\eta (\xi,x_0)\|, \quad \forall A_{x_0}\in \partial (-f)(x_0).$$
		Therefore
		$$f(x_0)-f(\xi)\geq_C A_{x_0} \eta (\xi,x_0)  - e \|\eta (\xi,x_0)\|,\quad \forall A_{x_0}\in \partial (-f)(x_0).$$
		Using \eqref{eq:thm1} and taking into account the fact that $\partial (-f )(x_0) = -\partial f (x_0)$ and $\eta(x_0, \xi)= -\eta(\xi, x_0)$, we obtain
		\begin{align*} A_{x_0} \eta(x_0,\xi)= (-A_{x_0}) \eta(\xi,x_0)&\leq_C f(x_0)-f(\xi)+ e \|\eta (\xi,x_0)\|\\
		&=f(x_0)-f(\xi)+ e \|\eta (x_0,\xi)\|\\
		&\leq_C 0,
		\end{align*}
		for any $A_{x_0}\in \partial f(x_0)$. Hence $\xi$ does not solve (MVVI).
	\end{proof}
	Furthermore, we prove that every solution of (SVVI) is still a local $(\eta,e)$-quasi efficient solution of (VOP) in the case of approximate pseudo invexity of type II.
	\begin{theorem}\label{thm3} Assume that $f$ is approximate pseudo $(\eta,e)$-invex of type II at $\xi \in X$. If $\xi$ is a solution of (SVVI), then $\xi$ is also a local $(\eta,e)$-quasi efficient solution of (VOP).
	\end{theorem}
	
	\begin{proof}
	By contrapositive, assume that for each
		$r> 0$, there is $x_0 \in B(\xi,r)$ satisfying
		\begin{equation}\label{eq:thm3}
		f(x_0)-f (\xi) \leq_C -e \|\eta(x_0,\xi)\|<_C 0.
		\end{equation}
		Since $f$ is approximate pseudo $(\eta,e)$-invex of type II at $\xi$, it follows that
		$$ A_{\xi} \eta(x_0,\xi) <_C -e \|\eta(x_0,\xi)\|, \quad \forall A_{\xi} \in \partial f(\xi).$$
		Applying equation \eqref{eq:thm3}, we get
		$$A_{\xi} \eta(x_0,\xi) \leq_C 0,\quad\forall A_{\xi} \in \partial f(\xi).$$
		Therefore $\xi$ is not a solution of (SVVI).
	\end{proof}
	Using similar arguments as in the proof of Theorem \ref{thm2}, we can also demonstrate that under approximate pseudo $(\eta,e)$-invexity of type II of the function $-f$, solutions to (MVVI) are also local e-quasi efficient solutions of (VOP) provided that $\eta$ holds the same condition.
	
	\section{Sufficient and necessary conditions for local quasi weak efficient solutions}
	
	In this section, we consider the weak formulations of Stampacchia and Minty vector variational inequalities as follows:\\
	
	(WSVVI) Find  $\xi\in X$ for which there exists no $x \in X$ satisfying
	$$A_{\xi} \eta(x,\xi)  <_C 0,\quad \forall A_{\xi} \in \partial f(\xi).$$
	
	(WMVVI) Find $\xi\in X$ for which there exists no $x \in X$ satisfying
	$$A_{x} \eta(x,\xi)  <_C 0,\quad \forall A_{x} \in \partial f(x).$$
	
	First, let us note that applying similar arguments as in the previous section, we can show that if $f$ is approximate pseudo $(\eta,e)$-invex of type I at a solution $\xi \in X$ of (WSVVI), then $\xi$ is also a local $(\eta,e)$-quasi weak efficient solution of (VOP). The converse implication is provided by the following theorem.
	\begin{theorem} Suppose that $\eta$ be is affine in the first argument with $\eta(x, x)=0$ for all $x\in X$. Assume also that $-f$ is approximate quasi $(\eta,e)$-invex of type II at $\xi\in X$. If $\xi$ is a local $(\eta,e)$-quasi weak efficient solution of (VOP), then $\xi$ is a solution of (WSVVI).
	\end{theorem}
	
	\begin{proof}
		Assume that $\xi$ is not a solution of (WSVVI). This means there is $x \in X$ such that for all $A_{\xi} \in \partial f(\xi)$, we have $A_{\xi} \eta(x,\xi)  <_C0.$ Hence
		\begin{equation}\label{eq:thm4}
		-A_{\xi} \eta(x,{\xi})  >_C  0.
		\end{equation}
		From $\partial (-f )(x) = -\partial f (x)$ we get $-A_{\xi} \in \partial (-f)({\xi}).$
		Since $-f$ is approximate quasi $(\eta,e)$-invex of type II at $\xi$, there is
		$\widetilde{r} > 0$ such that for any $x_0 \in B(\xi,\widetilde{r})$
		\begin{align}\label{eq:thm4-1} -A_{\xi} \eta(x_0,{\xi})  >_C 0 &\Rightarrow -f(x_0)-(-f({\xi}))>_C e\|\eta(x_0,\xi)\| \nonumber\\
		&\Rightarrow f(x_0)-f(\xi) <_C -e \|\eta(x_0,\xi)\| .
		\end{align}
		Let $r>0$ be arbitrary. We take $\overline{r}\leq min\{r,\widetilde{r}\}$ and $\lambda \in (0,1)$ so that $x_0=\lambda x+(1-\lambda)\xi \in B(\xi,\overline{r})\subseteq X$. We have
		$$A_{\xi} \eta(x_0,{\xi})=A_{\xi} \eta(\lambda x+(1-\lambda)\xi,{\xi})=\lambda A_{\xi} \eta(x,{\xi})$$ thanks to the assumptions on $\eta$.
		Thus, by using \eqref{eq:thm4}, we obtain $-A_{\xi} \eta(x_0,{\xi}) >_C  0$. Since $x_0\in B(\xi,\widetilde{r})$, then applying \eqref{eq:thm4-1}, the last inequality yields
		$$f(x_0)-f(\xi) <_C -e \|\eta(x_0,\xi)\| $$
		where $x_0\in B(\xi,r)$.
		Hence, $\xi$ cannot be an $(\eta,e)$-quasi weak efficient solution of (VOP).
	\end{proof}
	The following theorem illustrates when a solution of (WMVVI) is
	also a local $(\eta,e)$-quasi weak efficient solution of (VOP).
	\begin{theorem} Suppose that $-f$ is approximate pseudo $(\eta,e)$-invex of type I at $\xi$ with $\eta(x, \xi)+\eta(\xi, x)=0$ for all $x\in X$. If $\xi$ is a solution of (WMVVI), then $\xi$ is a local $(\eta,e)$-quasi weak efficient solution of (VOP).
	\end{theorem}
	\begin{proof}
		Assume that for each
		$r> 0$ there is $x_0 \in B(\xi,r)$ satisfying
		$$f(x_0)-f (\xi) <_C -e \|\eta(x_0,\xi)\|.$$
		Since $\eta(x_0, \xi)=-\eta(\xi, x_0)$, we obtain
		\begin{align*}-f(\xi)-(-f )(x_0) <_C -e \|\eta(\xi,x_0)\|.
		\end{align*}
		As the function $-f$ is approximate pseudo $(\eta,e)$-invex of type I at $\xi$, it follows that
		$$A_{x_0} \eta(\xi,x_0) <_C -e \|\eta(\xi,x_0)\|, \quad \forall A_{x_0} \in \partial (-f)(x_0).$$
		Using $\partial (-f )(x_0) = -\partial f (x_0)$ and $\eta(x_0, \xi)= -\eta(\xi, x_0)$, we obtain
		$$ A_{x_0} \eta(x_0,\xi)= (-A_{x_0}) \eta(\xi,x_0)
		<_C -e \|\eta(\xi,x_0)\|
		<_C 0,
		$$
		for any $A_{x_0}\in \partial f(x_0)$. 
	\end{proof}
	The next result specifies that vector critical points represent sufficient optimality conditions. Let us first recall their definition.
	\begin{Definition} \cite{Gut2016OL} A vector critical point of $f$ is a feasible point $\xi\in X$ so that the system $\mu^T A_{\xi} = 0$ admits a solution $\mu>_C0$ for some $A_{\xi} \in \partial f(\xi) $.
	\end{Definition}
	The gist of the above definition is that a vector critical point $\xi$ means $0\in \mu^T \partial f(\xi)$ has at least a positive solution $\mu$. Note that in case of a scalar-valued objective function $f$, a critical point $\xi$ is a solution to the inclusion problem $0\in \partial f(\xi)$.
	
	\begin{theorem}\label{thm6} Assume $f$ is approximate pseudo $(\eta,e)$-invex of type I at $\xi \in X$. If $\xi$ is a vector critical point of $f$, then $\xi$ is a local $(\eta,e)$-quasi weak efficient solution of (VOP).
	\end{theorem}
	
	To prove this result, we need to use the following theorem of the alternative.
	\begin{lemma} \cite{Bazaraa2006wiley} (Gordan's Theorem) If $A$ is a $n\times m$ matrix, then we have either
		\begin{enumerate}
			\item $Ax<_C0\quad$ for some $x \in \mathbb{R}^m$; or
			\item $A^Ty=0,\; y\geq_C 0\quad$ for some nonzero solution $y\in \mathbb{R}^n$;
		\end{enumerate}
	but not both.
	\end{lemma}
	
	\begin{proof}[Proof of Theorem \ref{thm6}]
		Assume that for each
		$r> 0$, there is $x_0 \in B(\xi,r)$ satisfying
		\begin{align*}
		f(x_0)-f (\xi) < -e \|\eta(x_0,\xi)\|.
		\end{align*}
		Since $f$ is approximate pseudo $(\eta,e)$-invex of type I at $\xi$, we obtain
		$$A_{\xi} \eta(x_0,\xi) <_C 0, \quad \forall A_{\xi} \in \partial f(\xi).$$
		By applying Gordan's Theorem we deduce that there is no $\mu>_C0$ such that $\mu^T A_{\xi} = 0$ for all $A_{\xi} \in \partial f(\xi)$. We conclude $\xi$ is not a vector critical point of $f$.
	\end{proof}
	
\begin{Remark}
Theorem \ref{thm6} improves Lemma 3.1 in \cite{MishraUpadhyay2013Positivity} since the approximate pseudo convexity of type I has been
weakened by the approximate pseudo $(\eta,e)$-invexity of type I.
\end{Remark}

\section{Example}
In this section, we illustrate the obtained results by an example.

Consider the following vector optimization problem:

$$\min f(x) :=( f_1(x) , f_2(x) ),\quad \mbox{s.t.} \quad x \in X,$$
where,
$$
f_1(x)=\begin{cases}  -x^3-x^2+5x & x\geq 0 \\
x^3+6x & x<0;
\end{cases}
$$

and

$$
f_2(x)=\begin{cases}  x^2-2x & x\geq0 \\
-x^2-3x & x<0,
\end{cases}
$$

$X=\mathbb{R}$, $C=\mathbb{R}^2_+$ and $\eta(x,y)=x-y$ for all $x,y \in X$.

The Clarke subdifferential of $f$ at $x \in X$ is given by
$$\partial f(x)= \begin{cases}
\{(-3x^2-2x+5,2x-2)^T\} & x>0 \\
co\{(5;-2)^T,(6;-3)^T\}=\{(5k_1+6k_2;-2k_1-3k_2)^T\} & x=0 \\
\{(3x^2+6,-2x-3)^T\} & x<0
\end{cases}\ ,$$
where $k_1\geq0$, $k_2\geq0$ such that $k_1+k_2=1$.

For any $e=(e_1,e_2)$ s.t $0 < e_i<1$ with $i\in\{1,2\}$, we prove that there exists $r=\frac{1}{2} >0$ such that, for all $x,y \in B(x_0, r)$, $x_0=0$, one has
		$$f(x) - f(y) <_C 0  \quad \Rightarrow \quad A_y\eta(x,y) + e\|\eta(x,y)\| <_C 0,
		\; \forall A_y \in \partial f(y),$$
Hence, $f$ is approximate pseudo $(\eta,e)$-invex of type II at $x_0$.

On the other hand, let $\xi=0$.

Since for any $x\in X\setminus\{\xi\}$ and for all $k_1\geq0$, $k_2\geq0$ such that $k_1+k_2=1$ one has
$$ A_\xi\eta(x,\xi)= x(5k_1+6k_2;-2k_1-3k_2)^T\nleq_C 0\ ,$$

Therefore, $\xi=0$ solves (SVVI).\\

Now, since $f$ is approximate pseudo $(\eta,e)$-invex of type II at $\xi$, then, by Theorem \ref{thm3}, $\xi=0$
should be a local $(\eta,e)$-quasi efficient solution of (VOP). Indeed, for
$0 < e<1$,  we have for all $x\in B(\xi,r)\setminus\{\xi\}$ with $r> 0$
$$f(x)-f(\xi)+e\|\eta(x,\xi)\| =f(x)+e|x|=\begin{cases}  (-x^3-x^2+5x+ex,x^2-2x+ex)^T & x> 0 \\
(x^3+6x-ex,-x^2-3x-ex)^T & x<0
\end{cases} \ ,$$
which means that
$$f(x)-f(\xi)+e\|\eta(x,\xi)\| \nleq_C 0,$$
Therefore, $\xi$ is a local $(\eta,e)$-quasi efficient solution of (VOP).\\

\section{Conclusion}

We have considered two generalized types of quasi efficient solutions to nonsmooth vector optimization problems. Using new generalized invexity assumptions we have provided necessary and sufficient conditions of optimality for these solutions. More precisely, we have shown that Stampacchia vector variational inequalities present sufficient optimality conditions when the multiobjective function $f$ satisfies weak forms of approximate invexity. On the other hand, we have obtained necessary optimality conditions in terms of Minty vector variational inequalities given similar approximate invexity assumptions on the function $-f$. Finally, we have proven that vector critical points represent sufficient optimality conditions as well for approximate pseudo invex functions. It is worth mentioning that we have used the original definition of Clarke's generalized Jacobian instead of the Cartesian product of  Clarke's subdifferentials of the scalar-valued function components as done in~\cite{Mishra2006NA,GuptaMishra2018Optimization,MishraUpadhyay2013Positivity,Gut2016OL}. Since our generalized invexity contains as special cases previous generalized convexity conditions that were provided in~\cite{NgaiLucThera2000JNCA,BhatiaGupta2013OptimLett,noor2006JMAA}, the presented theorems in this paper extend many corresponding results in the literature like~\cite{MishraUpadhyay2013Positivity} for example.

\end{document}